\documentclass[12pt]{article}
\usepackage{amsmath,amsthm,amssymb}
\usepackage{cite}
\usepackage[OT2, T1]{fontenc}
\usepackage[russian, english]{babel}
\usepackage{comment}
\usepackage{rotating}
\usepackage{graphicx}
\usepackage{pdflscape}
\usepackage{url}
\usepackage[latin1]{inputenc}
\usepackage{tikz}
\RequirePackage[colorlinks,citecolor=blue,urlcolor=blue]{hyperref}
\usepackage[symbol]{footmisc}

\newcommand\co{{\rm co}}

\newcommand\Hypergeo{{\rm Hypergeo}}
\newcommand\indicator{{\mathbb I}}

\newcommand\given{\, \vert \, }
\newcommand\E{{\mathbb E}}

\newtheorem{theorem}{Theorem}[section]
\newtheorem{corollary}{Corollary}[section]
\newtheorem{lemma}{Lemma}[section]

\newtheorem{remark}{Remark}[section]

\newcommand\prob{{\mathbb P}}
\begin{document}
\begin{center}
{\huge \bf Random multi-hooking networks}

\bigskip
{\large Kiran R. Bhutani\footnote{Department of Mathematics, The Catholic University of America, Washington, D.C. 20064, U.S.A.; Email: \url{bhutani@cua.edu} } \qquad Ravi Kalpathy\footnote{Department of Mathematics, The Catholic University of America, Washington, D.C. 20064, U.S.A.; Email: \url{kalpathy@cua.edu} } \qquad Hosam Mahmoud\footnote{Department of Statistics, The George Washington University, Washington, D.C. 20052, U.S.A.; Email: \url{hosam@gwu.edu}}}

\medskip
\end{center}
\bigskip\noindent
\section*{Abstract} 
We introduce a broad class of multi-hooking networks, 
wherein multiple copies of  
a seed are hooked at each step 
at random locations, and 
the number of copies 
follows a predetermined building sequence of numbers.

We analyze the degree profile in random multi-hooking networks by tracking two kinds of node degrees---the local average degree of a specific node over time and the global overall average degree in the graph. The former experiences phases and the latter is 
invariant with respect to the type of building sequence and is somewhat similar to the average degree in the initial seed. We also discuss the expected number of nodes of the smallest degree. 

Additionally, we study distances in the network through the lens of the average total path length, the average depth of a node, the eccentricity of a node, and the diameter of the graph. 

\bigskip
\noindent{\bf AMS subject classifications:} Primary:  05C82, 
90B15; 
{\bf }Secondary: 60C05, 
05C12. 

\medskip\noindent
{\bf Keywords:} Hooking networks, random graph,
degree profile,  recurrence.

\section{Introduction}
Trees have long been in the focus of research on random graphs. 
The classic types, such as those that appear in data structures~\cite{Brown,Drmota,Knuth,Mah92} and
digital processing~\cite{Briandais,Fredkin, Szpankowski}, grow incrementally, one node at a time. In more recent
times, authors considered more complex types of random graphs grown
by adjoining entire graphs to a growing
network~\cite{Bahrani,Bhutani1, ChenChen,Holmgren,Colin1,Colin2,Gittenberger,Mohan,profile,SP,Samorodnitsky,VanderHofstad}. We consider a growing network model in which the number of components attached at a stage follows a predetermined building sequence of numbers.

Societies and social networks grow 
and change over time in multiple random ways, which include growth patterns that add ``components'' at each step.  
Networks grown by adding components reflect
these dynamics better than networks evolving on single node additions. 
One can embed a 
graph in a predetermined growth structure leading to multiple scenarios of growing networks. 

In this paper, we develop a model where networks grow by hooking multiple copies
of the seed at multiple nodes of the growing network 
chosen in a random fashion
and study the theoretical and statistical properties of the networks so generated.
\section{The building sequence}
We assume that a network grows by attaching a number of components at each step to the existing structure, which starts with $\tau_0 \ge 2$ vertices. In the next subsection, we give a formal definition. Here, we only
say a word on the number of components added at each step.
After $n$ steps of growth, the number of components attached
to obtain the next network
is $k_n$, a
predetermined sequence of nonnegative numbers.
\bigskip\noindent
\subsection{Regularity conditions}
Let $\tau_0\ge 2$. This represents the number of nodes in a building
block (a seed). We grow the network by adding a number of copies
of the seed at places called latches. At each latch, a designated vertex
in the seed (called the hook) is fused with the latch. 
A formal definition of this process is given in the sequel.

We shall consider adding $k_{n}\ge 1$ copies of the seed to construct
the $(n+1)$st network,
under the following {\em regularity conditions}:
\begin{itemize}
\item [(R1)]  $k_n \le \tau_0 + (\tau_0 -1) \sum_{i=0}^{n-1} k_i$.
\item [(R2)]  $\lim_{n \to \infty}  \frac {k_n} {\sum_{i=0}^n k_i} = a \in [0, 1]$.
\item [(R3)]  $\lim_{n \to \infty} \frac {\tau_0} {k_n} = b \ge 0$.
\end{itemize}
A sequence of nonnegative integers $\{k_n\}_{n=0}^\infty$ satisfying 
(R1)--(R3) is called
a {\em building sequence}.
Condition (R1) is to guarantee the feasibility of choosing latches. At no point
in time does the process require more 
(distinct)
latches than the number of nodes
existing in the network.
Conditions (R2)--(R3) facilitate the existence of limits
for properties of interest and
expedite finding their values.
Note that $a=0$ and $b=0$ are both allowed.
For instance, for a constant sequence $k_n = k \in \mathbb N$, we have $a=0$,
and $b = \tau_0/ k >0$, whereas when $k_n = n+1$, 
we have both $a= 0$ and $b=0$.  

Regularity conditions (R1)--(R3) are not too restrictive and the class covered
by the investigation remains very broad. The examples
that come up in practice satisfy these regularity conditions. For example,
at one extreme
the building sequence $k_n=1$ builds networks of linear growth,
including trees. At the other extreme,
the case of equality in Condition (R1) builds a deterministic network
where the entire vertex set is chosen at each step (a take-all model); such extremal case
 grows the network exponentially fast.
\section{The multi-hooking network}
A network grows as follows.
We start with a connected {\em seed} graph $G_0$ 
with vertex set of size $\tau_0$ and edge set of size $\eta$.  
One of the vertices in the 
seed is designated as a {\em hook} (vertex $h$). When a copy of the seed is adjoined
to the network, it is the seed's hook that latches into that larger graph. The hooking
is accomplished by fusing together the hook and a
{\em latch} 
(vertex) chosen from the network.

 At step $n$, $k_{n-1}$ copies of the seed are hooked into the graph,
 $G_{n-1} = (V_{n-1}, {\cal E}_{n-1})$,
with vertex set $V_{n-1}$ and edge set ${\cal E}_{n-1}$,
  that exists at time $n-1$. To complete the $n$th hooking step,
we sample $k_{n-1}$ latches from the graph $G_{n-1}$. The selection
mechanism can take a number of forms, such as choosing distinct hooks
as opposed to allowing repetitions.

 We use the notation $|A|$ for the cardinality of a set $A$.
We consider a {\em uniform} model that selects $k_{n-1}$ {\em distinct}
nodes in the network, with all $|V_{n-1}| \choose k_{n-1}$ subsets being
equally likely. In the language of statistics, this boils down to
{\em sampling without replacement}.

Figure~\ref{Fig:network} illustrates  a seed and a network grown from it in three steps under the building sequence $k_n = n+1$. So, $G_1$ grows by choosing
a latch from $G_0$ (the starred node in $G_0$), 
$G_2$ grows by choosing the
two starred nodes from~$G_1$,
and 
$G_3$ grows by choosing the
three  
starred
nodes from~$G_2$.
The networks in Figure~\ref{Fig:network} have loops and multiple edges,
as we do not restrict the study to
{\em simple} graphs.
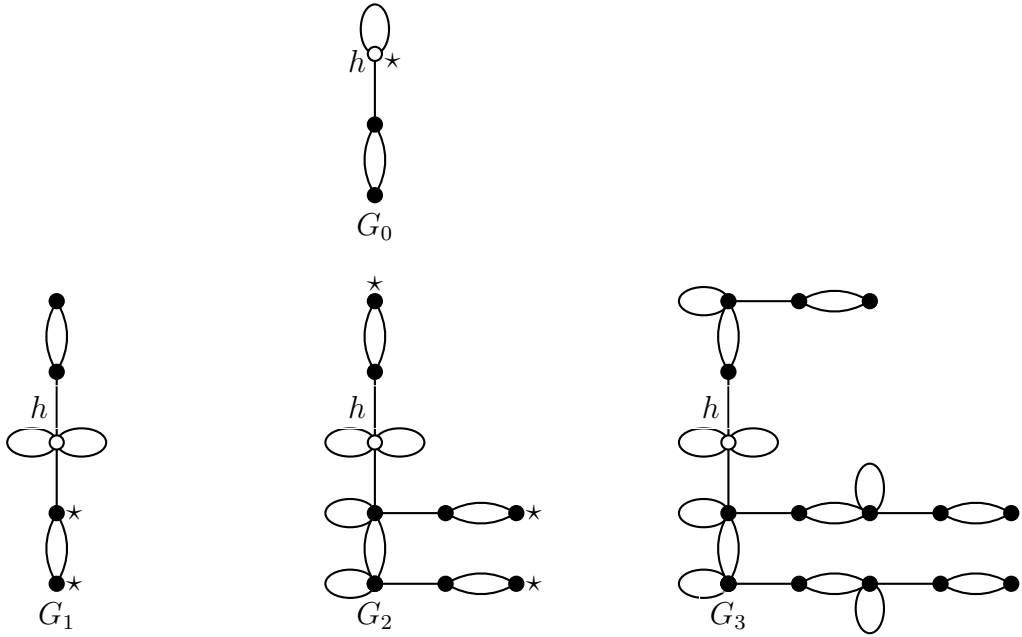
\begin{figure}[thb]
\begin{center}
\begin{tikzpicture}[scale=0.47]
\node[draw=white] at (-0.5, 14.8) {$h$};
\node[draw=white] at (0.5, 14.8) {$\star$};
\draw [thick] (0,15.7) ellipse (0.4 and 0.7);
\draw [thick] (0.7,4) ellipse (0.7 and 0.4);
\draw [thick] (-0.7,4) ellipse (0.7 and 0.4);
\draw [thick] (10.7,4) ellipse (0.7 and 0.4);
\draw [thick] (9.3,4) ellipse (0.7 and 0.4);
\draw [thick] (9.3,0) ellipse (0.7 and 0.4);
\draw [thick] (9.3,2) ellipse (0.7 and 0.4);
\node[draw=white] at (0, 10.1) {$G_0$};
\node[draw=white] at (-9, -0.9) {$G_1$};
\node[draw=white] at (0,  -0.9) {$G_2$};
\node[draw=white] at (10,  -0.9) {$G_3$};

\node[draw=white] at (-8.5, 0) {$\star$};
\node[draw=white] at (-8.5, 2) {$\star$};
\node[draw=white] at (4.5, 0) {$\star$};
\node[draw=white] at (4.5, 2) {$\star$};
\node[draw=white] at (0, 8.5) {$\star$};
\draw [thick, fill=black] (0,13) circle [radius=0.2];
\draw [thick, fill=black] (0,11) circle [radius=0.2];

\coordinate (A) at (0,15);
\coordinate (B) at (0,13);
\coordinate (C) at (0,11);
\draw [thick] (0,11) to [bend right] (0,13);
\draw [thick] (0,11) to [bend left] (0,13);
\draw [thick] (A)--(B);
\draw [thick, fill=white] (0,15) circle [radius=0.2];
\coordinate (A1) at (-9,0);
\coordinate (B1) at (-9,2);
\coordinate (C1) at (-9,4);
\coordinate (D1) at (-9,6);
\coordinate (E1) at (-9,8);
\draw [thick] (B1)--(D1);
\draw [thick] (-9,0) to [bend left] (-9,2);
\draw [thick] (-9,0) to [bend right] (-9,2);
\draw [thick] (-9,6) to [bend left] (-9,8);
\draw [thick] (-9,6) to [bend right] (-9,8);
\draw [thick, fill=black] (-9,8) circle [radius=0.2];
\draw [thick, fill=black] (-9,6) circle [radius=0.2];

\draw [thick, fill=black] (-9,2) circle [radius=0.2];
\draw [thick, fill=black] (-9,0) circle [radius=0.2];
\draw [thick] (-8.3,4) ellipse (0.7 and 0.4);
\draw [thick] (-9.7,4) ellipse (0.7 and 0.4);
\draw [thick, fill=white] (-9,4) circle [radius=0.2];
\node[draw=white] at (-9.5, 5) {$h$};
%
\coordinate (A2) at (0,0);
\coordinate (B2) at (0,2);
\coordinate (C2) at (0,4);
\coordinate (D2) at (0,6);
\coordinate (E2) at (0,8);
\coordinate (F2) at (2,0);
\coordinate (G2) at (2,2);
\draw [thick] (A2)--(F2);
\draw [thick] (B2)--(D2);
\draw [thick] (0,0) to [bend left] (0,2);
\draw [thick] (0,0) to [bend right] (0,2);
\draw [thick] (B2)--(G2);
\draw [thick] (0,6) to [bend right] (0,8);
\draw [thick] (0,6) to [bend left] (0,8);
\draw [thick] (2,0) to [bend left] (4,0);
\draw [thick] (2,0) to [bend right] (4,0);
\draw [thick] (2,2) to [bend left] (4,2);
\draw [thick] (2,2) to [bend right] (4,2);
\draw [thick, fill=black] (0,8) circle [radius=0.2];
\draw [thick, fill=black] (0,6) circle [radius=0.2];
\draw [thick, fill=white] (0,4) circle [radius=0.2];
\node[draw=white] at (-0.5, 5) {$h$};
%
\draw [thick, fill=black] (0,2) circle [radius=0.2];
\draw [thick, fill=black] (0,0) circle [radius=0.2];
\draw [thick, fill=black] (0,0) circle [radius=0.2];
\draw [thick, fill=black] (0,0) circle [radius=0.2];
\draw [thick, fill=black] (2,2) circle [radius=0.2];
\draw [thick, fill=black] (4,2) circle [radius=0.2];
\draw [thick, fill=black] (4,0) circle [radius=0.2];
\draw [thick, fill=black] (2,0) circle [radius=0.2];
\draw [thick] (-0.7,0) ellipse (0.7 and 0.4);
\draw [thick] (-0.7,2) ellipse (0.7 and 0.4);

\coordinate (A3) at (10,0);
\coordinate (B3) at (10,2);
\coordinate (C3) at (10,4);
\coordinate (D3) at (10,6);
\coordinate (E3) at (10,8);
\coordinate (F3) at (12,0);
\coordinate (G3) at (12,2);
\coordinate (H3) at (16,0);
\draw [thick] (A3)--(F3);
\draw [thick] (B3)--(D3);
\draw [thick] (12,0) to [bend left] (14,0);
\draw [thick] (12,0) to [bend right] (14,0);
\draw [thick] (B3)--(G3);
\draw [thick] (10,6) to [bend right] (10,8);
\draw [thick] (10,6) to [bend left] (10,8);
\draw [thick] (12,2) to [bend left] (14,2);
\draw [thick] (12,2) to [bend right] (14,2);
\draw [thick] (10,0) to [bend right] (10,2);
\draw [thick] (10,0) to [bend left] (10,2);
\draw [thick, fill=black] (10,8) circle [radius=0.2];
\draw [thick, fill=black] (10,6) circle [radius=0.2];
\draw [thick, fill=white] (10,4) circle [radius=0.2];
\node[draw=white] at (9.5, 5) {$h$};
%
\draw [thick, fill=black] (10,2) circle [radius=0.2];
\draw [thick, fill=black] (10,0) circle [radius=0.2];
\draw [thick, fill=black] (10,0) circle [radius=0.2];
\draw [thick, fill=black] (10,0) circle [radius=0.2];
\draw [thick, fill=black] (12,2) circle [radius=0.2];
\draw [thick, fill=black] (14,2) circle [radius=0.2];
\draw [thick, fill=black] (14,0) circle [radius=0.2];
\draw [thick, fill=black] (12,0) circle [radius=0.2];
\draw [thick, fill=black] (16,0) circle [radius=0.2];
\draw [thick, fill=black] (18,0) circle [radius=0.2];
\draw [thick] (16,0) to [bend left] (18,0);
\draw [thick] (16,0) to [bend right] (18,0);
\draw [thick] (14,0) to (16,0);
\draw [thick, fill=black] (16,2) circle [radius=0.2];
\draw [thick, fill=black] (18,2) circle [radius=0.2];
\draw [thick] (16,2) to [bend left] (18,2);
\draw [thick] (16,2) to [bend right] (18,2);
\draw [thick] (14,2) to (16,2);
\draw [thick] (14,2.7) ellipse (0.4 and 0.7);
\draw [thick] (14,-0.7) ellipse (0.4 and 0.7);
\draw [thick, fill=black] (12,8) circle [radius=0.2];
\draw [thick, fill=black] (14,8) circle [radius=0.2];
\draw [thick] (12,8) to [bend left] (14,8);
\draw [thick] (12,8) to [bend right] (14,8);
\draw [thick] (10,8) to (12,8););
\draw [thick] (9.3,8) ellipse (0.7 and 0.4);
\end{tikzpicture}
\end{center}
\caption{A seed (top) with a hook and three networks grown from it (second row) under
the building sequence $k_n=n+1$. The white vertices in the network $G_1$, $G_2$, and $G_3$ represent the reference vertex.}
\label{Fig:network}
\end{figure}
\subsection{Notation}
\label{Subsec:notation}
The notation $\Hypergeo(t, r, s)$ stands for the hypergeometric random variable associated with the random sampling of $s$ objects out of a total of $t$ objects, of which
$r$ objects are of a special type. So, the hypergeometric random variable counts the number of special objects in the sample.

It is customary to call the cardinality of the vertex set of a graph the {\em order
of the graph} and reserve the term {\em size of the graph to} the cardinality of the set of edges in the graph.
Let $V_n$ be the set of vertices of the graph~$G_n$, and ${\cal E}_n$ be the set
of edges of that graph.
Thus, the  seed $G_0 =(V_0, {\cal E}_0)$ is a connected graph with the set $V_0$ of vertices and
the set ${\cal E}_0$ of edges.

Let $\tau_n$ be the order of the graph at age $n$. Hence,
the cardinality
of the vertex set $V_0$ of the seed is $|V_0|=\tau_0$.
The $n$th hooking step adds $k_{n-1}$ copies of the seed at
$k_{n-1}$ distinct latches chosen uniformly at random from $G_{n-1}$. Each copy
contributes $\tau_0 -1$ new vertices to the network. The reason
for subtracting 1 is the absorption of the hook.
This gives the recurrence
\begin{equation}
\tau_n =  \tau_{n-1}  + k_{n-1} (\tau_0-1) .
\label{Eq:taun}
\end{equation}
Unwinding this recurrence,  we obtain
\begin{equation}
\tau_n = (\tau_0-1)\sum_{i=1}^n k_{i-1}  + \tau_0.
\label{Eq:taun2}
\end{equation}

We use the notation $\deg(v)$ to denote the degree of node $v$ in a given 
graph, and we set $h^* = \deg(h)$.
\subsection{Useful limits}
\label{Subsec:useful}
By the regularity conditions, we can argue from~(\ref{Eq:taun2})
that
$$\frac {\tau_n} {k_{n-1}} = (\tau_0-1)\frac 1 {k_{n-1}}\sum_{i=0}^{n-1} k_i  +
   \frac {\tau_0} {k_{n-1}} \to \frac {\tau_0-1} a+b =: \gamma.$$

Reorganize~(\ref{Eq:taun}) as
$$1 =  \frac {\tau_{n-1}} {\tau_n}  + \frac {k_{n-1}}{\tau_n} (\tau_0-1) .$$
to find the limit
$$  \frac {\tau_{n-1}} {\tau_n}  = 1 - \frac {k_{n-1}}{\tau_n} (\tau_0-1)
   \to 1- \frac{(\tau_0-1)} {\gamma}.$$
\section{A degree profile of the network}
\label{Sec:deg}
Various aspects of the degrees of nodes in a network are of interest in different contexts. For example, in the language of epidemiology, the degree of a node may be a useful representation of a highly infective person. From a health policy point of view, having knowledge about the degrees in conjunction with other graph parameters may help in identifying hot spots that trigger outbreaks and may be useful in controlling and mitigating the contagion. 
In the context of a social network, the degree of a node may represent the popularity and social skills of the person represented by the node. 

Equally interesting are the global overall average degree in the entire graph (where we look at all the nodes), the local degree of a specific node during its temporal evolution, and the number of nodes of the smallest degree. We deal with the average behavior of each of these in a separate subsection. The different aspects of the degree complete a profile of the graph. 
\subsection{Evolution of the degree of a specific node}
Suppose a node appears for the first time at step $j$. What will become of
its degree at step $n$? At step $j$, several copies are added. To avoid a heavy
notation identifying the time of appearance $j$, the copy number, which node within the copy to be tracked, and $n$, we use a simpler notation that needs 
only $j$ and $n$, for after all nodes of the same degree in the seed have the same distribution over time.
\begin{theorem}
\label{Theo:deg}
Suppose $\{k_n\}_{n=0}^\infty$ is a building sequence of the family of 
graphs $\{G_n\}_{n=1}^\infty$. 
Let $X_{j:n}$ be the degree
of a node at time $n$ that had appeared for the first time at step $j$. If initially its degree (in the seed) is $\delta$, then we have
 $$\E[X_{j:n}]   =\delta + h^*  \sum_{i=j}^{n-1} \frac {k_i} {\tau_i} =   \frac {h^*a} {(1-a)(\tau_0 -1) + ab}(n-j) + o(n-j) + O(1).$$
\end {theorem}    
\begin{proof}
Suppose a node $v$ appears at time $j$ for the first time. So, it belongs to one 
of the copies adjoined to the graph at that time.   
As the graph evolves, in any single step the degree  
of $v$ can increase, if it is one of the nodes selected as latches in that step;
otherwise its degree stays put, and when it does increase, it goes up
by $h^* = \deg(h)$, the degree of the hook in the seed. This gives 
rise to a recurrence:
$$X_{j:n} = X_{j:n-1} + h^* \indicator_{n-1}(v), $$
where $\indicator_{n-1}(v)$ is an indicator of the event of choosing
$v$ among the $k_{n-1}$ latches of that step of growth. 
On average, we have 
$$\E[X_{j:n}] = \E[X_{j:n-1}] + h^* \frac {{\tau_{n-1}-1\choose k_{n-1}-1}} {{\tau_{n-1}\choose k_{n-1}}} = \E[X_{j:n-1}]  +h^*  \frac {k_{n-1}} {\tau_{n-1}}. $$
 
 Unwinding the recurrence, we obtain the exact average:
$$\E[X_{j:n}] =\delta + h^*  \sum_{i=j}^{n-1} \frac {k_i} {\tau_i}. $$
By the limits in Subsection~\ref{Subsec:useful},  we obtain
$$\E[X_{j:n}]  = O(1) + o(n-j) 
      + \frac {h^*a} {(1-a)(\tau_0 -1) + ab}(n-j).$$ 
\end{proof} 
\begin{remark} 
If $a=0$  the 
$\E[X_{j:n}]$ is only $o(n-j)$.   
\end{remark}
\begin{remark} 
Consider the case $a > 0$.
The average in Theorem~\ref{Theo:deg} indicates
that the degree of a specific node experiences phases. The degree of
a node in the early phase   
with $j = j(n) = o(n)$ grows linearly with its age in the network. When $j(n)\sim \rho n$, for $0 < \rho < 1 $, 
we still get a linear growth, but the coefficient of linearity 
is attenuated to $(1-\rho)\, \frac {h^*a} {(1-a)(\tau_0 -1) + ab}$.  
At $\rho=1$, we have $\E[X_{j:n}] = o(n)$.
\end{remark}
\begin{remark} 
If $a=0$, we can only assert that $\E[X_{j:n}]  =o(n-j) + \delta$. 
In this case,
 a finer analysis is needed to identify the
leading order of the average degree of a node that appears at time
$j$. For instance,
in the case of a tree grown from the complete graph $K_2$,
we have $\delta = 1, h^*=1, k_n =1$, and $
a=0$. The exact formula in this case yields 
$$\E[X_{j:n}] =1 +  \sum_{i=j}^{n-1} \frac 1 {i+2}=H_{n+1} -
             H_{j+1} + 1.$$
Whence, we have the phases
$$\E[X_{j:n}] 
    \sim \begin{cases}
           \ln n,                         &\mbox{if \ } j \mbox{\ is fixed};\\
           \ln \frac {n} {j(n)},     &\mbox{if \ } j(n)\to \infty \ \mbox {and \ } 
                  j(n) = o(n);\\
           \ln \frac 1 {\rho},     &\mbox{if \ } j(n) \sim \rho n, ~~ 0 < \rho < 1;\\
           1,   &\mbox{if \ } j(n) = n - o(n).  
     \end{cases}$$          
\end{remark}
\subsection{The overall average degree}
The main result about the overall average degree in the graph is developed in 
this section.
The result is expressed in terms of $\eta$, the number of edges in the seed graph.
\begin{theorem}
Suppose $\{k_n\}_{n=0}^\infty$ is a building sequence of the family of 
graphs $\{G_n\}_{n=1}^\infty$.  
Let $Y_n$ be the degree
of a randomly chosen node in the graph $G_n$ at age~$n$.
We have
 $$\lim_{n\to\infty}\E[Y_n]  =\frac {2\eta}{\tau_0-1}.$$
\end {theorem} 
\begin{proof}
Upon hooking 
$k_{n-1}$ copies of the seed to $k_{n-1}$ distinct nodes of  $G_{n-1} = (V_{n-1}, {\cal E}_{n-1})$,  we add $\eta k_{n-1}$ edges to the graph. 
Therefore, we have 
$$|{\cal E}_n| = |{\cal E}_{n-1}| + \eta k_{n-1}.$$
This recurrence has the solution
$$|{\cal E}_n| = \eta\Big(1+\sum_{i=0}^{n-1} k_i\Big).$$
Using the classical 
First Theorem of Graph Theory, we obtain
$$\sum_{v\in V_n} \deg (v) = 2{|\cal E}_n| = 2 \eta\Big(1+\sum_{i=0}^{n-1} k_i\Big).$$
Scaling the equation by $\tau_n$, 
we 
get     
$$\frac 1 {\tau_n} \sum_{v\in V_n}  \deg (v) = \frac {2\eta} {\tau_n}\Big(1+\sum_{i=0}^{n-1} k_i\Big).$$
Taking limits, and using equation (2), we obtain    
$$\lim_{n\to\infty} \E[Y_n]= 0+\lim_{n\to\infty}2\eta\, \frac{\sum_{i=0}^{n-1} k_i}
   {\tau_n} = \frac {2\eta}{\tau_0-1 }.$$  
\end{proof}
\begin{remark} The average degree in the seed is $2\eta/\tau_0$.
For any building sequence, the asymptotic
average degree in the graph is $2\eta / (\tau_0-1)$,
only slightly higher than the average degree in the initial seed. This should be anticipated because 
the additions introduce a number of copies of the seed, each of
which has the degree properties of the seed with the hook eliminated.   
\end{remark}
\subsection{Nodes of the smallest degree}
\label{Sec:smallestdegree}
We study only the nodes of the smallest degree.  Let $d^*$ be the smallest degree in
the seed.  Note that the smallest admissible
degree in the graph is~$d^*$. After the network grows, the smallest degree in it 
may be $d^*$ or higher.
Let $X_n$ be the number of nodes of degree~$d^*$ at time~$n$.
Thus, $X_0$ is the number
of nodes of degree~$d^*$ in the seed.
Later graphs can have more nodes of degree~$d^*$.
The seed in Figure~\ref{Fig:network} has $d^* = 2$, 
and $X_0 = 1,
X_1 = 2, X_2=3$, and $X_3 =3$.
\subsubsection{Stochastic recurrence}
In the evolution at step $n$, we hook $k_{n-1}$ copies of the seed to the graph~$G_{n-1}$.
Let $A_0$ be the event $\deg(h) = d^*$ and $\mathbb I_{A_0}$
be an indicator that assumes value 1, if $\deg(h) = d^*$,
otherwise, it assumes the value 0.
A latch of degree $d^*$ in the sample
will have a higher degree (namely, its degree goes up to $d^* + \deg(h)$) in $G_n$. 
So, we lose such vertices in the count of $X_n$. If the hook degree is $d^*$, every hooked
copy contributes only $X_0-1$ vertices of degree $d^*$.

For the case when $k_{n-1} = 1$ and the latch is $\ell$, 
the change from $X_{n-1}$ to~$X_n$ for the four cases can be seen as shown in the table below:

\vskip.4in
\begin{tabular}{c|c|c}
 & ${\deg(\ell)} = d^*$  & ${\deg(\ell)} \neq d^*$ \\
   \hline
 $ \deg(h) = d^*$ & $(X_{n-1} -1) + (X_0 -1)$ & $ X_{n-1} + (X_0-1) $ \\
    \hline
$ \deg(h) \neq d^*$ & $(X_{n-1} -1) + X_0$ & $ X_{n-1} + X_0$
\end{tabular}

\bigskip
\noindent Thus,
 for any value of $k_{n-1}$, the count $X_n$ therefore satisfies a (conditional) stochastic recurrence:
 \label{Eq:martingale}
\begin{align}
X_n &=  X_{n-1}  +  (X_0 - 1 -\mathbb I_{A_0})\, \Hypergeo(\tau_{n-1},
   X_{n-1}, k_{n-1})\nonumber \\
   &\qquad {} + (X_0 - \mathbb I_{A_0})\big(k_{n-1} -\Hypergeo(\tau_{n-1} , X_{n-1}, k_{n-1})\big).
 \label{Eq:Xn}
\end{align}
\subsubsection{The average proportion of nodes of degree $d^*$}
Take (conditional) expectation of~(\ref{Eq:Xn}) to get
\begin{align}
\E\big[X_n\given G_{n-1}\big]
&= X_{n-1}  +  (X_0 - 1 -\mathbb I_{A_0}) \frac {X_{n-1}}
    {\tau_{n-1}} k_{n-1} \nonumber \\
       &\qquad {} + (X_0 - \mathbb I_{A_0})\Big(k_{n-1} - \frac {X_{n-1}} {\tau_{n-1}} k_{n-1}\Big)
          \nonumber \\
   &= \Big(1   -\frac {k_{n-1}} {\tau_{n-1}} \Big) X_{n-1}
            + (X_0 - \mathbb I_{A_0}) k_{n-1}.
\label{Eq:EXn}
\end{align}
\begin{theorem}
\label{Theo:critical}
Suppose $\{k_n\}_{n=0}^\infty$ is a building sequence of the family of 
graphs $\{G_n\}_{n=1}^\infty$, starting from a seed with $X_0$ nodes of the smallest degree $d^*$. Let $X_n$ be the number  of vertices of this degree in the graph
after $n$ steps of evolution according to the building sequence.
We have
$$\E[X_n] = (X_0-\mathbb I_{A_0})
         \sum_{i=1}^n  k_{i-1}\prod_{j=i+1}^n
         \Big(1 - \frac {k_{j-1}} {\tau_{j-1}}\Big)
    +X_0\prod_{j=1}^n \Big(1 - \frac {k_{j-1}} {\tau_{j-1}}\Big). $$
Subsequently, the average proportion converges to a limit independent of the limits $a$ and $b$; namely we have the convergence
$$\E\Big[ \frac {X_n} {\tau_n}\Bigr] \to \frac {X_0-\mathbb I_{A_0}} {\tau_0}.$$
\end{theorem}
\begin{proof}
Taking a double expectation of~(\ref{Eq:EXn}) yields
\begin{equation}
\E[X_n]  =   \Big(1 - \frac {k_{n-1}} {\tau_{n-1}}\Big) \, \E[X_{n-1}]
                  +  (X_0 - \mathbb I_{A_0})\, k_{n-1} .
\label{Eq:uncond}
\end{equation}
This recurrence equation
is of the standard linear form
\begin{equation}
y_n = g_n y_{n-1} + h_n,
\label{Eq:standard}
\end{equation}
with solution
\begin{equation}
y_n = \sum_{i=1}^n h_i\prod_{j=i+1}^n g_j + y_0 \prod_{j=1}^n g_j. \label{Eq:standardsol}
\end{equation}
So, the sought solution for the average of the number of nodes
of degree $d^*$ (for $n\ge 1$) is
$$\E[X_n] = (X_0-\mathbb I_{A_0}) \sum_{i=1}^n  k_{i-1}
   \prod_{j=i+1}^n \Big(1 - \frac {k_{j-1}} {\tau_{j-1}}\Big)
    + X_0 \prod_{j=1}^n \Big(1 - \frac {k_{j-1}}
    {\tau_{j-1}}\Big). $$
    The strategy for the asymptotic part of the statement is two-fold: We prove the existence of a limit (under any building sequence) for the proportion  from the exact solution. We then find the value of the limit from the recurrence under the mild regularity conditions imposed on the building sequence.

 First, express the expected proportion as
\begin{equation}
\E\Big[\frac {X_n}{\tau_n}\Big] =
 (X_0-\mathbb I_{A_0})\sum_{i=1}^n  \frac {k_{i-1}}
    {\tau_n}\prod_{j=i+1}^n \Big(1 - \frac {k_{j-1}} {\tau_{j-1}}\Big)
    + \frac {X_0} {\tau_n}
    \prod_{j=1}^n \Big(1 - \frac {k_{j-1}} {\tau_{j-1}}\Big); 
    \label{Eq:Kiran}
\end{equation}
at $i=n$, the first product does not exist, and is taken to be 1, as usual.
Let
  $$c_n = \sum_{i=1}^n \frac{k_{i-1}} {\tau_n}\prod_{j=i+1}^n  \Bigl( 1 -
       \frac {k_{j-1}}
             {\tau_{j-1}}\Bigr)
                 = \sum_{i=1}^{n-1} \frac{k_{i-1}} {\tau_n}\prod_{j=i+1}^n  \Bigl( 1 -
       \frac {k_{j-1}}
             {\tau_{j-1}}\Bigr)+ \frac {k_{n-1}}{\tau_n}; $$
We manipulate this to turn it into a recurrence as follows:
\begin{align*}
c_{n+1} 
     &= \sum_{i=1}^n \frac{k_{i-1}} {\tau_{n+1}}\prod_{j=i+1}^{n+1}  \Bigl( 1 -
            \frac {k_{j-1}}
             {\tau_{j-1}}\Bigr) +  \frac {k_n}{\tau_{n+1}}\\
     &= \frac {\tau_n}{\tau_{n+1}}\Big(1 - \frac {k_n}{\tau_n} \Big)\sum_{i=1}^n \frac{k_{i-1}} 
               {\tau_n}\prod_{j=i+1}^n  \Bigl( 1 -
            \frac {k_{j-1}}
             {\tau_{j-1}}\Bigr) +  \frac {k_n}{\tau_{n+1}}   \\
       &=  \Big(\frac {\tau_n - k_n}{\tau_{n+1}} \Big)
                     c_n+  \frac {k_n}{\tau_{n+1}}   \\    
       &=  \Big(\frac {\tau_{n+1} - (\tau_0-1)k_n - k_n}{\tau_{n+1}} \Big)
                     c_n+  \frac {k_n}{\tau_{n+1}}    \\    
       &= \Big(\frac {\tau_{n+1} - \tau_0 k_n }{\tau_{n+1}} \Big)
                     c_n+  \frac {k_n}{\tau_{n+1}}   
\end{align*}
Rearrange the recurrence in the form
$$c_{n+1} - \frac 1 {\tau_0} =c_n - \frac {\tau_0 k_n }{\tau_{n+1}}\, c_n
                     + \frac {k_n}{\tau_{n+1}}   -\frac 1 {\tau_0} = \Big(c_n - \frac 1 {\tau_0}\Big) \Big (\frac {\tau_{n+1}-\tau_0 k_n}{\tau_{n+1}}\Big),$$
leading to the inequality
\begin{align*}
\Big|c_{n+1} - \frac 1 {\tau_0} \Big|
&\le\Big|c_n - \frac 1 {\tau_0}\Big|  \,\Big |\frac {\tau_{n+1}-\tau_0 k_n + k_n}{\tau_{n+1}}\Big| \\
&= \Big|c_n - \frac 1 {\tau_0} \Big|  \,\frac {\tau_n}{\tau_{n+1}}\\
&\le \Big|c_{n-1} - \frac 1 {\tau_0} \Big| \,\frac {\tau_n}{\tau_{n+1}} \times \,\frac {\tau_{n-1}}{\tau_n}\\
&\ \ \vdots\\
&\le \Big|c_0 - \frac 1 {\tau_0} \Big| \,\frac {\tau_n}{\tau_{n+1}} \times \,\frac {\tau_{n-1}}{\tau_n}
    \times \cdots\times  \frac {\tau_0}{\tau_{1}}.
\end{align*}
Noting that the sum in $c_n$ is empty at $n=0$, we have $c_0 = 0$ and the bounds simplify to
$0 \le |c_n - 1 /\tau_0| \le 1/\tau_n$. So, both inferior and superior limits of
$|c_n - 1 /\tau_0|$ are equal to $0$, which furnishes the existence of a limit
for $c_n$ equal to $1/\tau_0$, too.

As for the remainder part
$$r_n := \frac {X_0}{\tau_n}\prod_{j=1}^n \Bigl( 1 - \frac {k_{j-1}}
             {\tau_{j-1}}\Bigr),
$$
in~(\ref{Eq:Kiran}), it
clearly converges to 0, as $\tau_n$ is increasing, and the product is bounded from above by~1.
Plugging in the limits $\lim_{n\to \infty} c_n= 1/\tau_0$ and $\lim_{n\to \infty} r_n = 0$ in~(\ref{Eq:Kiran}), we reach the conclusion that
$$\lim_{n\to \infty} \E\Bigl[ \frac {X_n} {\tau_n}\Bigr] = \frac {X_0- \mathbb I_{A_0}} {\tau_0}.$$
\end{proof}
\begin{remark}
In the case when the hook is not of the smallest degree $d^*$, we have
$\E[X_n/\tau_n] \to X_0/\tau_0$. The initial proportion
of nodes of the smallest degree in the seed is preserved on average
in 
larger subsequent graphs.
\end{remark}
\begin{remark}
In the case when the hook is  of the smallest degree $d^*$, we have
$\E[X_n/\tau_n] \to (X_0-1)/\tau_0$. The long-term proportion
of nodes of the smallest degree is less than the proportion
of nodes of degree $d^*$ in the seed.
\end{remark}
\begin{remark}
In the case when the hook is the only node of the smallest degree~$d^*$ in the seed, we have $X_0=1$, and
$\E[X_n/\tau_n] \to 0$, for all $n\ge 1$. Indeed, the degree~$d^*$ disappears
after the first latching at the initial hook and never reappears.
\end{remark}
\begin{remark}
The limit in Theorem~\ref{Theo:critical} is more than just an ultimate value in the take-all case. In this case, it is the actual value for each $n\ge0$, which can be seen from the recurrence. The only term that does not vanish is the last term in sum,  yielding $(X_0 - I_{A_0})  k_{n-1}/ \tau_n =  (X_0 - I_{A_0})  / \tau_0$.
\end{remark}
\section{Distances in the network}
We measure node distances in $G_n$ relative to
a reference point (vertex). We take the reference to be the hook of $G_0$.
We look at two (related) kinds of distances: The total path length and
the average distance in the graph.
Let the nodes of the $n$th graph be labeled with the numbers $1,2, \ldots, \tau_n$,
with 1 being reserved for the reference vertex
and the rest of the nodes are arbitrarily
assigned distinct numbers from the set $\{2, \ldots, \tau_n\}$.
The depth of a node in the network is its distance from the reference vertex (i.e.,
the length of the shortest path from the node to the reference vertex measured in the number of edges).
We denote the depth of the $i$th node in the $n$th network by $\Delta_{n,i}$
The {\em total
path length} is the sum of all depths; namely it is
$$T_n = \sum_{i=1}^{\tau_n} \Delta_{n,i}.$$
For instance, the networks $G_0, G_1, G_2$, and $G_3$ in Figure~\ref{Fig:network} have total path lengths 
$T_0=3, T_1 = 6$,  $T_2 = 18$, and $T_3 = 45$,
 respectively. 
\subsection{Average total path length}
As the network grows, at step $n$, a sample of size $k_{n-1}$ latches
is chosen from~$G_{n-1}$ to grow $G_{n-1}$ 
into $G_n$. Suppose these latches are 
$\ell_1, \ldots, \ell_{k_{n -1}}$
at depths
$d_1, \ldots, d_{k_{n-1}}$. In view of the absorption
of the hooks of the added graphs,
a copy's hook fused at the $j$th latch adds $\tau_0-1$
nodes, which appear in $G_n$ at depths equal to their distance from the hook of the copy
translated by an additional distance from the latch to the reference vertex. So, collectively,
the vertices of the copy hooked to $\ell_j$ increase the total path length by
$T_0   + (\tau_0-1)
d_j$.
We have a conditional recurrence:
$$\E\big[T_n \given G_{n-1}, d_1, \ldots, d_{k_{n-1}}\big]
             = T_{n-1} + \sum_{j=1}^{k_{n-1}} (T_0
                        +  (\tau_0-1)   d_j).$$
Averaging over the graphs and the choices of the latches within,
we get
\begin{equation}
\E[T_n] = \E[T_{n-1}] + T_0 k_{n-1} + (\tau_0-1) \sum_{j=1}^{k_{n-1}}
               \E[ \Delta_{n-1,\ell_j}].
\label{Eq:Ln}
\end{equation}
\begin{lemma}
\label{Lm:TPL}
\begin{align*}
 \sum_{j=1}^{k_{n-1}}
               \E[ \Delta_{n-1,\ell_j}]
           =  \frac {k_{n-1}} {\tau_{n-1}} \, \E[T_{n-1}].
\end{align*}
\end{lemma}
\begin{proof}
Condition on the event $\ell_1 =i_1, \ldots,
                       \ell_{k_{n-1}} = i_{k_{n-1}}$, to get
\begin{align*}
      \sum_{j=1}^{k_{n-1}}
               \E[ \Delta_{n-1,\ell_j}] &= \sum_{j=1}^{k_{n-1}}
              \sum_{1\le i_1< i_2 < \cdots< i_{k_{n-1}} \le \tau_{n-1}}
              \E\big[ \Delta_{n-1,\ell_j}\given d_1 =i_1, \ldots
                       d_{k_{n-1}} = i_{k_{n-1}}\big]\\
            &\qquad\qquad \qquad\qquad {} \times \prob(d_1 =i_1, \ldots,
                       d_{k_{n-1}} = i_{k_{n-1}}).
\end{align*}
The subsets of size $k_{n-1}$ latches that appear in a sample of vertices
from $G_{n-1}$ are all equally likely, and we get
\begin{align*}
       &\sum_{j=1}^{k_{n-1}}
               \E[\Delta_{n-1,\ell_j}] \\
               &\qquad =
         \frac 1 {{{\tau_{n-1}} \choose {k_{n-1}}}}\sum_{j=1}^{k_{n-1}}
            \E\Big[ \Big(\sum_{1\le i_1< i_2 < \cdots < i_{k_{n-1}} \le \tau_{n-1}}   \Delta_{n-1,\ell_j}\Big)\, \big |\, \ell_1 =i_1, \ldots
                       \ell_{k_{n-1}}= i_{k_{n-1}}\Big].
\end{align*}
Let us write out the inner sum in expanded form:
\begin{align*}
        &  \Delta_{n-1,1} +   \Delta_{n-1,2} +\cdots + \Delta_{n-1,k_{n-1}} \\
        & \qquad {} +  \Delta_{n-1,1} + \Delta_{n-1,2} +\cdots
                 + \Delta_{n-1,k_{n-1}-1} +
                 \Delta_{n-1,k_{n-1}+1}\\
             &  \qquad\ \, {}  \vdots\\
             & \qquad {} +  \Delta_{n-1,\tau_{n-1} - k_{n-1}+1}
                +  \Delta_{n-1,\tau_{n-1} - k_{n-1}+2}+\cdots
                 + \Delta_{n-1,\tau_{n-1}}.
\end{align*}
Upon a reorganization collecting similar terms, we get
$$ {\tau_{n-1 -1} \choose k_{n-1}-1}
       (\Delta_{n-1,1} +   \Delta_{n-1,2} + \cdots + \Delta_{n-1,\tau_{n-1}})
          = {\tau_{n-1 -1} \choose k_{n-1}-1} T_{n-1}.$$
Plugging this expression in the expectation, we proceed to
$$\sum_{j=1}^{k_{n-1}}
               \E[\Delta_{n-1,\ell_j}] =
                 \frac {{\tau_{n-1 -1} \choose k_{n-1}-1}}
                     {{\tau_{n-1} \choose k_{n-1}}}\, \E[T_{n-1}] = \frac {k_{n-1}}
                     {{\tau_{n-1}}} \, \E[T_{n-1}].$$
\end{proof}
\begin{theorem}
\label{Theo:TPL}
Suppose $\{k_n\}_{n=0}^\infty$ is a building sequence of the family of 
graphs $\{G_n\}_{n=1}^\infty$, starting from a seed of total path length $T_0$. Let $T_n$ be the total path length
after $n$ steps of evolution according to the building sequence.
We have
$$\E[T_n] = T_0\tau_n\Big(\sum_{i=1}^n  \frac {k_{i-1}} {\tau_i}
    + \frac 1 {\tau_0}\Big). $$
\end{theorem}
\begin{proof}
By Lemma~\ref{Lm:TPL} and the recurrence~(\ref{Eq:Ln}), we have a recurrence
for the average total path length:
\begin{align*}
\E[T_n] &= \E[T_{n-1}] +  T_0 k_{n-1}+ \frac {(\tau_0-1) k_{n-1}} {\tau_{n-1}} \, \E[T_{n-1}] \\
                  &= \Big(1 + \frac {(\tau_0-1)  k_{n-1}} {\tau_{n-1}}
                  \Big)\E[T_{n-1}] + T_0  k_{n-1}.
\end{align*}
Again, the recurrence is of the standard form~(\ref{Eq:standard})
with the solution~(\ref{Eq:standardsol}). In the specific case at hand,
this solution is
$$\E[T_n] = T_0\sum_{i=1}^n  k_{i-1}\prod_{j=i+1}^n \Big(1 + \frac {(\tau_0-1){k_{j-1}} } {\tau_{j-1}}\Big)
    +T_0 \prod_{j=1}^n \Big(1 + \frac {(\tau_0 -1)k_{j-1}} {\tau_{j-1}}\Big). $$
The recurrence~(\ref{Eq:taun}) on the order of the graph simplifies
the solution into telescopic products
$$\E[T_n] = T_0\sum_{i=1}^n  k_{i-1}\prod_{j=i+1}^n \frac
{\tau_j} {\tau_{j-1}}
    +T_0 \prod_{j=1}^n \frac {\tau_j} {\tau_{j-1}}
       = T_0 \tau_n \Bigl(\sum_{i=1}^n  \frac{k_{i-1}}{\tau_i}
    + \frac 1 {\tau_0}\Bigr). $$
\end{proof}
\subsection{Average depth}
Theorem~\ref{Theo:TPL} provides a benchmark for the calculation of
the average depth. Let the depth of a randomly selected node in the $n$th
network be $D_n$.
\begin{corollary}
$$\E[D_n]
       = T_0 \sum_{i=1}^n  \frac{k_{i-1}}{\tau_i}
    + D_0. $$
\end{corollary}
\begin{proof}
Given a specific development leading to $G_n$,
the average
depth in that graph is
$$\E\big[D_n \given G_{n}\big] = \frac {\Delta_{n,1} +\cdots + \Delta_{n,\tau_n}} {\tau_n} = \frac {T_n}{\tau_n}.$$
Upon taking expectation, it
follows that $\E\big[D_n] = \E[T_n] / \tau_n$. The form given
in the statement ensues from Theorem~\ref{Theo:TPL}.
\end{proof}
\begin{corollary}
\label{Cor:Dn}
Under the regularity conditions {\rm (R1)--(R3)}, we have the asymptotic equivalent
$$\E[D_n]
       = \frac {a T_0}{\tau_0 -1 + ab} \, n + o(n), \qquad \mbox{as \ } n \to \infty. $$
\end{corollary}
\begin{remark}
Corollary~\ref{Cor:Dn} is more useful when
 $$\lim_{n\to \infty }\frac {k_n} {\sum_{i=0}^n k_i}  = a\not = 0.$$
When $a=0$, as in the case of trees for example, one needs to sharpen
the argument to find the leading asymptotic term, as we do in some specific cases
below.
\end{remark}
\subsection{Distances under specific building sequences}
At one extreme, there is 
 the sequence of least possible growth ($k_n= 1$).
At the other extreme, we have a take-all model 
    ($k_n = \tau_n$) in which
all the nodes of  
$G_n$
are taken as latches for 
$\tau_{n}$ copies of
the seed.

In the case of $k_n = k$, for fixed $k\in \mathbb N$, of nearly the least growth,
the  average depth is
$$\E[D_n]
       = T_0 \sum_{i=1}^n  \frac k {\tau_i}
    +D_0. $$
    The limit $a$ in regularity condition (R2) is 0, and  Corollary~\ref{Cor:Dn} only tells us that $\E[D_n] =o(n)$. However, we can
    sharpen the asymptotic equivalence from the specific
    construction of the case.

Here, we have $\tau_i = (\tau_0 - 1) i + \tau_0$, which gives
  $$\E[D_n]
       = \frac {T_0 k} {\tau_0-1}
           \sum_{i=1}^n  \frac 1{ i + \frac {\tau_0} {\tau_0-1}}
     + D_0. $$
     In terms of the generalized harmonic 
numbers\footnote{Customarily, $H_n (0)$ is denoted by $H_n$.}
     $$H_n(x) = \frac 1 {1+x} +  \frac 1 {2+x}  +\cdots+\frac 1 {n+x}, $$
     the depth in the near-least-growth is compactly expressed as
      $$\E[D_n]
       = \frac {T_0 k} {\tau_0-1} H_n\Big(\frac {\tau_0}{\tau_0-1}\Big)
     + D_0 \sim \frac {T_0 k} {\tau_0-1} \ln n, \qquad \mbox {as \ } n\to\infty.$$
 \begin{remark} The case $k=1$ and $\tau_0=2$ grows a recursive tree.
 The seed is a rooted tree on two vertices, in which $T_0=1$ and $D_0 = \frac 1 2$. In this case, the average depth becomes
 $$\E[D_n]
       =H_n(2)
     + \frac 1 2 = H_{n+2} - 1 \sim  \ln n, \qquad \mbox {as \ } n\to\infty, $$
     which recovers a known result~\cite{Smythe}.
 \end{remark}
\begin{remark}
    At the other end of the spectrum, there is the take-all model, in which
     $k_i = \tau_i$, 
    leading at step
    $n$ to a graph of order  $\tau_n = \tau_0^{n+1}$.
    Here, the limit $a$ is $(\tau_0-1)/\tau_0$ and the limit $b$ is 0. According
    to Corollary~\ref{Cor:Dn}, we have $\E[D_n] \sim D_0 n$,
    as $n\to\infty$.
This asymptotic estimate can be sharpened as the case is amenable to
exact calculation:
$$\E[D_n] = T_0\sum_{i=1}^n  \frac {\tau_{i-1}} {\tau_i}
    + D_0  = T_0\sum_{i=1}^n  \frac 1 {\tau_0}  + D_0= D_0 (n+1)   . $$
\end{remark}
\section{Eccentricity}
\label{Sec:ecc}
The eccentricity $C(v)$ of a node $v$ in a graph $\cal G$ 
is the distance between~$v$ and a vertex farthest from $v$ in $\cal G$. The eccentricity is instrumental in constructing a notion of the diameter of a graph (extreme distances). We use the 
eccentricity of the hook and the various latches selected in $G_{n-1}$ to determine the diameter of the graph $G_n$.

The eccentricity is technically defined as follows.
If $\cal Q$ is a path in a graph~$\cal G$, we denote its length by $|\cal Q|$ (the number of edges in it). There can be several paths joining two vertices~$u$ and $v$ in $\cal G$, and the distance between~$u$ and $v$, denoted by $d(u,v)$, is the length of the shortest such path. That is, with ${\cal P}(u,v)$ denoting the collection of paths between $u$ and $v$, the distance between these two nodes is given by
$$ d(u,v) = \min_{{\cal Q}\in {\cal P}(u,v)} |{\cal Q}|.$$
The eccentricity $C(v)$ of a vertex $v$ in a graph with vertex set $\cal V$ is:
$$C(v) =  \max_{u\in \cal V} d(v,u)
   = \max_{u\in \cal V}  \min_{{\cal Q}\in {\cal P}(v,u)} |{\cal Q}|.$$
For instance, the eccentricity in Figure~\ref{Fig:network} of the reference vertex of $G_0$
is~2, of the reference vertex in $G_1$ is~2 as well, but of the reference vertex in $G_2$ is~4 
and becomes 6 in $G_3$.
\subsection {Eccentricity of a node in $G_n$}
\label{Subsec:latch}
The $k_{n-1}$ nodes selected as latches from the graph $G_{n-1}= (V_{n-1}, {\cal E}_{n-1})$ are vertices that play a key role in designing the network at stage $n$ and onward and contribute significantly in determining the diameter of the graph at the next stage.

As a node's eccentricity changes over time, its value at step $n$ in $G_n$
may be different from its value at step $n-1$ in $G_{n-1}$.
We need an eccentricity notation reflecting the possible change over time.
For that we use $C_n(v)$ to speak of the eccentricity of a vertex $v$ in $G_n$.

If $v \in V_n$ is a vertex in a copy of $G_0$ latched at a vertex 
$\ell_i \in V_{n-1}$, we express that by saying $v \in V_0^{co_i}$, otherwise we say $v \in V_{n-1}$. We now introduce some notation: 
\begin{enumerate}
\item $\mathbb{L}_n = \{\ell_1, \ell_2, \cdots \ell_{k_n}\}$ is the set of latches selected in the graph~$G_n$ to produce the graph $G_{n+1}$.
\item  $\widetilde{C}_n(v) =  C_n(v) \given G_{n-1}, \mathbb{L}_{n-1}$. This is the {\em conditional} eccentricity $C_n(v)$ of the node $v$ in the graph $G_n$, given $G_{n-1}$ and the $k_{n-1}$ latches in it.
\item For any $v \in V_{n-1}$, we define $d_v^{\#} = \max_{\ell_j \in \mathbb{L}_{n-1}} d(v, \ell_j)$. So,  $d_v^{\#}$ is the maximum distance from $v$ to the nodes
in $\mathbb{L}_{n-1}$.
\end{enumerate}
Also, in what follows we use the notation $\indicator_{\cal C}$ to indicate a predicate (condition)~$\cal C$. So, it is 1, when $\cal C$ holds,
and is 0, otherwise.
\bigskip
\begin{theorem}
\label{Theo:eccentricity}
Suppose $\{k_n\}_{n=0}^\infty$ is a building sequence of the family of 
graphs $\{G_n\}_{n=1}^\infty$. Let $v$ be a node in the graph $G_n$. 
Conditional upon the choice of the latches $\ell_1, \ldots, \ell_{k_{n-1}}$ in  
$G_{n-1}$, 
the eccentricity $C_n(v)$ is given by
$$\widetilde{C}_n(v)= 
{\small
\begin{cases}
                            \max\Big\{C_{n-1}(v), d_v^{\#}+C_0(h)\Big\},
                                     &\mbox{if \ } v\in V_{n-1};\\
                                     \\
                                     \\
                         \max\Big \{d(v,h) + C_{n-1}(\ell_i),    
                                  \\
                   \qquad {}  \max\big\{C_0(v)  ,  \big(d(v,h) + d_{\ell_i}^{\#} + C_0(h)\big)\indicator_{\{k_{n-1} >1\}}\big\},
                            &\mbox{if \ } v \in V_0^{co_i}.
                         \end{cases}}
$$
\end{theorem}
\begin{proof} 
The graph $G_n$ is obtained by attaching a copy of the seed $G_0$ at each of the latches $\ell_1, \ell_2, \cdots, \ell_{k_{n-1}}$ selected in the graph $G_{n-1}$.

We denote the vertex set of the $r$th copy of the seed, for
$r= 1, \ldots, k_{n-1}$, by $V_0^{\co_r}$.
We now compute the distance from a node $v$ to a vertex
$u$ in $G_n$ by considering the four cases:
\begin{center}
\begin{tabular}{c|c|c}
${}$ & $u \in V_{n-1}$  & $u \in V_0^{\co_j}$ \\
   \hline
 $ v\in V_{n-1}$ & $d(v,u)$ &$ d(v,\ell_j) + d(h, u)$ 
   \\
    \hline
$v \in V_0^{\co_i}$ & $ d(v, h) + d(\ell_i, u)$ 
   & $d(u,v) \indicator _{i=j} + (d(v,h)+d(\ell_i, \ell_j)+d(h,u))
    \indicator_{i\not =j}
$
\end{tabular}
\end{center}
\vskip 0.5cm
For a given $v\in V_n$, the maximum (over the range of $u$ and $j$) in each block is
\begin{center}
\begin{tabular} {c|c|c}
 &  $u \in V_{n-1}$  & $u$ in a copy \\
   \hline
 ${v\in V_{n-1}}$ & $C_{n-1}(v)$ &$ d_v^{\#} + C_0(h)$ \\
    \hline
$v\in V_0^{\co_i}$  & $d(v,h)+C_{n-1}(\ell_i)$ & $\max\big\{C_0(v)  ,  \big(d(v,h) + d_{\ell_i}^{\#} + C_0(h)\big)\indicator_{\{k_{n-1} > 1\}}\big\}$
\end{tabular}
\end{center}
\vskip 0.5cm
The result now follows.
\end{proof}
\begin{remark}
Suppose a vertex $\ell$ is chosen as a latch from $G_{n-1}${\color{red}.} 
From Theorem~\ref{Theo:eccentricity}, we pick up the top line
 and write
 $$\widetilde C_n(\ell) = \max \big\{C_{n-1}(\ell), \displaystyle d_\ell^{\#}+ C_0(h)\big\}.$$ 
If $k_{n-1} = 1 $, then $d_\ell^{\#}=0$, in which case we have $\widetilde C_n(\ell) = \max \big\{C_{n-1}(\ell), C_0(h)\big\}.$ 
\end{remark}
\section{Diameter of the graph $G_n$}
\label{Sec:diameter}
\newcommand\Russiand{\mbox{\foreignlanguage{russian}{D}}}

The diameter of a connected graph with vertex set $\cal V$ is the longest distance between any
two nodes in it~\cite{Chartrand}. That is, the diameter is the
maximum eccentricity,  $\max_{v\in \cal V} C(v)$. 
For example, the diameters of the graphs $G_0,G_1,G_2$ and $G_3$ 
in Figure~\ref{Fig:network} are respectively $2, 4, 6$ and $10$.

\bigskip
We now introduce some additional notation: 
\begin{enumerate}
\item  $\widetilde{\Russiand_{n}} =  \Russiand_{n} \given 
G_{n-1}, \mathbb{L}_{n-1}$. This is the {\em conditional} diameter $\Russiand_{n}$ 
of the graph~$G_n$ given~$G_{n-1}$ and the $k_{n-1}$ latches in $\mathbb{L}_{n-1}$ (see Subsection~\ref{Subsec:latch} for the definition of $\mathbb{L}_{n-1}$).
\item Only for $k_n > 1$, we define $q_{n} =  \max_{\ell, \tilde\ell \in \mathbb{L}_{n}} d(\ell, \tilde \ell) = \max_{\ell \in \mathbb{L}_{n}} \ell^{\#}$.
Thus,~$q_{n}$ computes the maximum distance between 
any two latches in~$G_{n}$.
\item $\alpha_{n} = \max_{\ell \in \mathbb{L}_{n}} C_{n}(\ell)$. 
Thus, $\alpha_{n}$ is the maximum eccentricity of a latch in $G_{n}$.
\end{enumerate}
\begin{theorem}
\label{Theo:diameter}
Suppose $\{k_n\}_{n=0}^\infty$ is a building sequence of the family of 
graphs $\{G_n\}_{n=1}^\infty$. The conditional diameter
$\widetilde{\Russiand_{n}} = \Russiand_n\given G_{n-1}, \mathbb{L}_{n-1} $ of a graph of age 
$n$ is given by
$$\widetilde{\Russiand_{n}}  = 
         \max \Big \{\Russiand_{n-1}, ~~\big(2C_0(h) + q_{n-1}\big) \indicator_{\{k_{n-1} > 1\}}, ~~ C_0(h) + \alpha_{n-1}  \Big\}. 
$$
\end{theorem}
\begin{proof}
The (conditional) diameter $\widetilde{\Russiand_{n}}$ of $G_n$ 
may remain the same 
 as the diameter $\Russiand_{n-1}$ of $G_{n-1}$,\footnote{In the graph $G_2$ in Figure~\ref{Fig:network}, if we pick the three latches at distances 2,3,4 from the top vertex,  the diameter of the graph  so obtained in step 3 will be equal 
 to $\Russiand_2$,  the diameter of~$G_2$.}
unless we can find longer paths in $G_n$. The latter case arises, 
 if  
\begin{itemize}
\item [(a)] There
is a pair~$x$ and $y$ 
of latches in $G_{n-1}$, and a pair of vertices (say~$u$  in the copy latched at $x$
and $v$ in the copy latched at $y$), such that
$d(u,x) + d(x,y) + d(y, v)>\Russiand_{n-1}$. The case can be, only if 
$k_{n-1} > 1$.
The longest such distance is obtained  by maximizing over $x,y,u$ and $v$ to obtain
$2C_0(h) + q_{n-1}$.\footnote{This situation occurs in the graph $G_3$ in Figure~\ref{Fig:network}.}
\item [(b)] 
Or, we can find a vertex
$u$ far enough from a latch $\ell$ in $G_{n-1}$ and another vertex $v$ in the copy latched at $\ell$ such that $d(u,\ell) + d(\ell, v)> \Russiand_{n-1}$.
The longest such distance is obtained by maximizing over $\ell, u$ and $v$
to obtain  $ \alpha_{n-1}+C_0(h) $.
\end{itemize}
The longest distance in the graph is the maximum of the three possibilities 
discussed.
\end{proof}
\begin{remark} 
\label{Rem:diam}
Consider the case where, at stage $n$ (for each $n\ge 1$), we pick among the $k_{n-1}$ latches two, say $\ell, \tilde \ell$ in $G_{n-1}$, such that $d(\ell, \tilde \ell)$ is the diameter of $G_{n-1}$. Note that this selection mechanism is no longer random in the sense discussed in all the preceding sections.
Let us call the diameter of the graph so constructed $\Russiand_{n-1}^*$. This is only possible if $k_{n-1} > 1$, for each~$n$. By arguments
similar to what we used in the proof of Theorem~\ref{Theo:diameter}, we get $\Russiand_{n}^* = 2C_0(h) + \Russiand_{n-1}^*$. Unwinding we get $\Russiand_{n}^* = 2n C_0(h) + \Russiand_{0}$.
\end{remark}
Remark~\ref{Rem:diam} shows that,
under this special hooking mechanism,
 the diameter~$\Russiand_{n}$ at step $n$ only requires the knowledge of the seed graph and $n$. It does not take into consideration how many latches were picked at stages $1$ through $n-1$ as long as there are two latches 
 $\ell, \tilde \ell$ picked at each stage such that $d(\ell, \tilde \ell)$ is maximum. 
\section*{Acknowledgment}
The authors would like to sincerely thank Dr.~Shaimaa M.\ Abd-Elaal,
Egyptian Ministry of Health and Population, for many
fruitful discussions that gave them a perspective on applications 
of the model proposed.
\section*{Competing interests}
The authors declare none.

\end{document}